\documentclass{svproc}
\usepackage{amsmath}
\usepackage{graphicx}
\usepackage{amsfonts,amssymb}
\usepackage{hyperref}
\hypersetup{
    colorlinks=true,        % ссылки будут цветными (а не в рамке)
    linkcolor=blue,         % цвет ссылок на утверждения
    citecolor=blue,         % цвет ссылок на библиографию
    urlcolor=blue,          % цвет ссылок на веб-адреса
    pdfborder={0 0 0},      % убирает рамки вокруг ссылок (на всякий случай)
    linktocpage=true,       % в оглавлении ссылкой становится номер страницы
}

\usepackage{url}

\def\Cl{\operatorname{Cl}}
\def\Conv{\operatorname{Conv}}
\def\B{\operatorname{B}}
\def\mpn{\operatorname{mpn}}
\def\smt{\operatorname{smt}}
\def\Li{\operatorname{Li}}
\def\Ls{\operatorname{Ls}}
\def\Lim{\operatorname{Lim}}
\def\Im{\operatorname{Im}}
\def\Ker{\operatorname{Ker}}

\begin{document}
\mainmatter     
\title{On the Existence of Minimal Parametric Networks}
\titlerunning{On the Existence of Minimal Parametric Networks}  
\author{Arsen Galstyan}
\institute{
School of Mathematics, Harbin Institute of Technology, Harbin, Heilongjiang 150001, China,\\
\email{ares.1995@mail.ru}.
}

\maketitle              % typeset the title of the contribution

\begin{abstract}

The paper continues the investigation of conditions for the existence of minimal networks in metric spaces. We introduce the notion of a \textit{compactly equipped} pseudometric space, in which every closed ball is equipped with a compact topology with respect to which the original pseudometric is lower semicontinuous. It is proved that these conditions are sufficient for the existence of minimal parametric networks of any type. This generalizes the corresponding theorems of Ivanov, Tropin and Tuzhilin on the existence of such networks in proper metric spaces. Moreover, this allows us to give both a more general formulation and a simpler proof of Bednov's theorem on sufficient conditions for the existence of minimal parametric networks in Banach spaces. We also show which hyperspaces and under what conditions inherit this compactly equipped property.

\end{abstract}

\section*{Introduction}

The problem of finding minimal (that is, shortest in the sense of length) networks connecting a given finite set of points (the \textit{boundary}) of a metric space is a classical problem in the calculus of variations and geometric optimization with numerous applications in science and engineering (see, e.g.,~\cite{Hwang},~\cite{Review},~\cite{Cheng}). Its most well-known formulation is the Steiner problem, which asks for a shortest connected network (metric graph) containing the boundary, with the possibility of adding extra vertices; these additional vertices are called \textit{interior} vertices.

Searching for such a minimal network can be extremely difficult because of the so-called ``combinatorial explosion'' --- the very rapidly growing (with the number of boundary points) number of possible ways to interconnect the original and added points, i.e., the number of tree structures that can connect the given boundary.

In order to reduce the combinatorial complexity, other formulations of the minimal connection problem are considered. One possibility is to fix the tree structure --- to make it a parameter (the so-called problem of finding a minimal parametric network; see, e.g.,~\cite{IT},~\cite{Branching},~\cite{Borodin},~\cite{Tropin}). The simplest structure in this case is a star tree, whose single interior vertex $y$ is connected to each boundary point. The length of such a tree is obviously the sum of the distances from the interior vertex to all boundary points.

From the very first formulation by Pierre Fermat of the problem of connecting the vertices of a triangle up to the works of the late 20th century, the Steiner problem was considered exclusively in finite-dimensional normed spaces (see, e.g.,~\cite{Hwang},~\cite{Du},~\cite{Bern},~\cite{Cieslik}). In such spaces, the existence of minimal trees and parametric networks is guaranteed by the classical results of~\cite{Hwang}.

It should be noted that guaranteeing the existence and actually finding the shortest network are different tasks. Melzak~\cite{Melzak} presented a computational geometry algorithm for constructing a minimal network for any finite set of points in the Euclidean plane. Zhitnaya, in recent works~\cite{Zhitnaya_1},~\cite{Zhitnaya_2} and~\cite{Zhitnaya_3}, showed that the problem of finding a shortest (including parametric) network in finite-dimensional spaces with $l_1$ and $l_2$ metrics reduces to the problem of finding the equilibrium position of some mechanical system.

In the work~\cite{IT}, Ivanov and Tuzhilin proved the existence of a minimal parametric tree (with an arbitrary fixed parameter) for any finite boundary on a connected complete Riemannian manifold. As noted by Tropin in his work~\cite{Tropin}, this result easily generalizes to the case of an arbitrary proper metric space. Also, Ivanov, Tropin and Tuzhilin~\cite{ITT} proved that the hyperspace of all nonempty compact subsets of a proper metric space is also a proper metric space, and therefore every finite boundary there is joinable by a minimal parametric network (note that the existence of minimal parametric networks of arbitrary fixed type in a metric space implies the existence of shortest networks).

However, in more general metric and even in Banach spaces the situation becomes more complicated: minimal (and even parametric) networks for finite boundaries may fail to exist, as shown by Garkavi~\cite{Garkavi}, Vesel\'{y}~\cite{Vesely}, Borodin~\cite{Borodin} and Strelkova~\cite{Bednov}. Furthermore, the completeness of a metric space is irrelevant for the question of the existence; see Proposition~\ref{prop:noncomplete}. Nevertheless, Bednov~\cite{Bednov} proved that if a Banach space admits a norm-one projector from its bidual, then every finite set of points can be joined by a minimal parametric network of arbitrary fixed type in that space.

On the other hand, Ambrosio and Tilli~\cite{Ambrosio}, building on ideas of Gromov~\cite{Gromov}, proved for any compact subset of a metric space the existence of a connected set of minimal one-dimensional Hausdorff measure containing it, under the following assumptions: it is sufficient that each closed ball admits a compact metrizable topology with respect to which the original metric is lower semicontinuous. However, their approach does not in general yield the structure of a finite graph, and in principle solves a slightly different problem of continuously (connectedly) joining a set of points by a set of minimal measure. In our problem, we do not require that the vertices of a minimal network can be continuously connected at all. The length of a network in our case is simply the sum of the lengths of its edges, where the length of an edge is the distance between its endpoints. A shortest network is accordingly a network of minimum possible length in the given space.

In the present paper, we continue the investigation of the existence of minimal parametric networks in various spaces. We prove (Theorem~\ref{thm:Ambrosio}) that the conditions of Ambrosio and Tilli (even without requiring the topology on the ball to be metrizable) are sufficient for the existence of minimal parametric networks of any type. A space satisfying such conditions has been named ``compactly equipped''. Note that a proper metric space is a special case of such a space.

Theorem~\ref{thm:Ambrosio} automatically allowed us to give a more general formulation and also a simpler proof of Bednov's result~\cite{Bednov}. Indeed, the idea of the proof of Theorem~1 in~\cite{Bednov} is to find a minimal parametric network in the bidual and then project it onto the original space. But in any dual space (in particular, in the bidual), the weak$^*$ topology is compact on every closed ball, and the metric is lower semicontinuous in this topology. Thus, in Bednov's theorem one can simply speak of a norm-one projector from a compactly equipped Banach space onto the original Banach space (Theorem~\ref{thm:gen}). It should be noted that whether the existence of such a projector implies a norm-one projector from the bidual remains an open question.

The other results of this paper concern the theory of minimal parametric networks in hyperspaces. We note that the theory of hyperspaces, which originated at the beginning of the 20th century thanks to the fundamental works of Hausdorff and Vietoris, has grown into an extensive field closely connected with topology, functional analysis and geometric measure theory (see, e.g.,~\cite{Nadler},~\cite{Blackburn}). Kelley's famous 1942 paper~\cite{Kelley} generated sustained interest in hyperspaces, and subsequent research showed (see, e.g., the book~\cite{SetAn}) that such spaces provide a natural framework for the study of set-valued analysis and variational problems. Note that in recent years, Ivanov, Tuzhilin, Tropin and Galstyan wrote a series of papers investigating minimal parametric networks in hyperspaces (see Definition~\ref{dfn:hyperspace}) of compact subsets of finite-dimensional normed spaces (see \cite{Tropin}, \cite{ITT}, \cite{Tropin_2}, \cite{Gals_1}, \cite{Gals_2}).

In this paper, we prove that if the original pseudometric space is compactly equipped and the corresponding compact topology $\tau$ on a closed ball $B$ is Hausdorff, then the hyperspace of nonempty $\tau$-closed subsets of the ball $B$ is also compactly equipped (Theorem~\ref{thm:hyperspace}). Consequently, in such hyperspaces, minimal parametric networks of any type exist for any finite boundary.

From Theorem~\ref{thm:hyperspace}, it follows that the hyperspace of nonempty weak$^*$-closed and bounded subsets of a dual space is also compactly equipped (Theorem~\ref{thm:hyp_tau_comp}), and therefore minimal networks exist there as well.

Furthermore, by Mazur's theorem~\ref{thm:mazur}, in a normed space the closure of any convex set coincides with its closure in the weak topology, which in the reflexive case coincides with the weak$^*$ topology on that space. It follows that the hyperspace of nonempty closed bounded convex subsets of a reflexive space is also compactly equipped (Theorem~\ref{thm:conv_hyperspace}). Hence, by Theorem~\ref{thm:Ambrosio}, for any finite boundary in such a hyperspace, a minimal parametric network of any type exists.

Thus, the main results of this paper are Theorem~\ref{thm:Ambrosio}, Theorem~\ref{thm:gen}, Theorem~\ref{thm:hyperspace}, Theorem~\ref{thm:hyp_tau_comp}, and Theorem~\ref{thm:conv_hyperspace}. These theorems demonstrate that the approach based on compactly equipped spaces is a natural tool for proving the existence of minimal networks in various spaces. It allowed us to unify and extend previously known existence theorems.

The paper is organized as follows. In Section~\ref{Definitions}, we collect all the necessary definitions and auxiliary statements concerning $r$-neighborhoods (Subsection~\ref{sec:Neighborhoods}), the Hausdorff distance (Subsection~\ref{sec:Haus}), hyperspaces (Subsection~\ref{sec:hyperspace}), the Vietoris topology (Subsection~\ref{sec:Vietoris}), weak topologies (Subsection~\ref{sec:weak_top}), reflexive spaces (Subsection~\ref{sec:ref_space}), lower semicontinuous functions (Subsection~\ref{sec:semicont}), complemented spaces (Subsection~\ref{sec:complemented}) and also we rigorously introduce the notions of a minimal parametric network (Definition~\ref{dfn:mpn}) and a minimal Steiner tree (Definition~\ref{dfn:mst}). In Section~\ref{Main}, we present the main existence results for shortest networks: first for compactly equipped pseudometric spaces (Subsection~\ref{sec:pseudo}), then for Banach spaces (Subsection~\ref{sec:Banach}), and finally for hyperspaces over pseudometric (Subsection~\ref{sec:hyp_pseudo}), dual (Subsection~\ref{sec:hyp_dual}), and reflexive spaces (Subsection~\ref{sec:hyp_ref}).

\subsection*{Acknowledgments}

The author wishes to express his sincere gratitude to Professor A. A. Tuzhilin and Professor A. O. Ivanov for their valuable remarks.

The work of Arsen Galstyan was carried out at Harbin Institute of Technology with the support of the Postdoctoral Research Start-up Funds (Funding card number: AUGA5710026125).

\section{Necessary Definitions and Statements}\label{Definitions}
%

%%%%%%%%%%%%%%%%%%%%%%%%%%%%%%
\subsection{Pseudometric Space}
%%%%%%%%%%%%%%%%%%%%%%%%%%%%%%

\begin{definition}
Let $X$ be an arbitrary set. A function $d\colon X\times X\to [0,\infty)$ is called a \textit{pseudometric} if for all $x,y,z\in X$ the following conditions are satisfied$:$
\begin{enumerate}
\item[$(1)$] $d(x,x)=0;$
\item[$(2)$] $d(x,y)=d(y,x)$ $($symmetry$);$
\item[$(3)$] $d(x,z)\le d(x,y)+d(y,z)$ $($triangle inequality$).$
\end{enumerate}
In contrast to a metric, a pseudometric does not require that $d(x,y)=0$ implies $x=y;$ distinct points may be at zero distance.
The set $X$ together with such a function $d$ is called a pseudometric space.
\end{definition}

We emphasize that any metric space is a pseudometric space. For convenience, the distance between two points $a, b$ of a pseudometric space $X$ will be denoted by $|a\, b|$ instead of $d(a, b)$.

%%%%%%%%%%%%%%%%%%%%%%%%%%%%%%
\subsection{Definition and Properties of $r$-Neighborhoods}\label{sec:Neighborhoods}
%%%%%%%%%%%%%%%%%%%%%%%%%%%%%%

\begin{definition}\label{dfn:one}
For any subset $A$ of a pseudometric space $X$ and any point $p\in X$, the \textit{distance from $p$ to $A$} is the value
$$
|p\, A|=\inf\bigl\{|p\, a|\colon a\in A\bigr\}.
$$
In particular, for empty $A$ we have $|p\, A|=\infty$.
\end{definition}

\begin{definition}\label{dfn:balls}
Let $A$ be a subset of a pseudometric space $X$ and $0\le r<\infty$. The subsets
$$
B_r(A)=\{p\in X\colon |p\, A|\le r\},\ \ U_r(A)=\{p\in X\colon |p\, A|<r\}
$$
are called, respectively, the closed and open ball centered at $A$ with radius $r$ or the closed and the open $r$-neighborhood of $A$.
\end{definition}

In the case $A = \{a\}$, the notations $B_r\bigl(\{a\}\bigr)$ and $U_r\bigl(\{a\}\bigr)$ will be replaced for brevity by $B_r(a)$ and $U_r(a)$, respectively. 

The next lemma is standard, see~\cite[p.~22]{TuzLections}. The proof of this lemma does not rely on the property $|x\, y|=0 \Rightarrow x=y$; hence the lemma remains valid in pseudometric spaces.

\begin{lemma}\label{sum_0}
Let $A$ be a nonempty subset of a pseudometric space $X$, and $r, r' \in [0, \infty)$. Then $B_r\bigl(B_{r'}(A)\bigr)\subset B_{r+r'}(A)$.
\end{lemma}

%%%%%%%%%%%%%%%%%%%%%%%%%%%%%%
\subsection{Hausdorff Distance}\label{sec:Haus}
%%%%%%%%%%%%%%%%%%%%%%%%%%%%%%

\begin{definition}\label{distance_H}
Let $X$ be a pseudometric space. For any nonempty $A, B\subset X$, we define the Hausdorff distance $d_H$ as follows:
$$
d_H(A, B) = \inf\bigl\{r:A\subset B_r(B),\,B\subset B_r(A)\bigr\}.
$$
\end{definition}

It is also well known that there is an equivalent definition of the Hausdorff distance $d_H$ in metric spaces (see, for example, the work~\cite[Proposition~5.1]{TuzLections}). The proof of the equivalence of these definitions does not use the property $|x\, y|=0 \Rightarrow x=y$ in any way; therefore, the equivalence also holds in pseudometric spaces.

\begin{definition}\label{distance_H2}
$$
d_H(A, B) = \max\bigl\{\sup\limits_{a\in A} |a\, B|,\, \sup\limits_{b\in B} |b\, A|\bigr\}.
$$
\end{definition}

Notice that $d_H$ is non-negative, finite for bounded $A$ and $B$, symmetric, and satisfies the triangle inequality, but can be zero for $A\ne B$. The following result, which also holds in pseudometric spaces, is well known, see, e.g.,~\cite[Proposition~7.3.3]{Burago}.

\begin{proposition}
Let $X$ be a pseudometric space. Then
\begin{itemize}
\item $d_H\bigl(A, \Cl(A)\bigr) = 0$ for any nonempty $A\subset X;$
\item If $A$ and $B$ are closed subsets of the space $X$ and $d_H(A, B) = 0$, then $A = B.$
\end{itemize}
\end{proposition}

%%%%%%%%%%%%%%%%%%%%%%%%%%%%%%
\subsection{Hyperspace}\label{sec:hyperspace}
%%%%%%%%%%%%%%%%%%%%%%%%%%%%%%

Let $X$ be a topological space, and let $\tau$ be another topology on $X$. We denote by $\Cl A$ the closure of $A$ in the original topology, and by $\Cl_{\tau} A$ the closure in the topology $\tau$. Define 
\begin{eqnarray*}
  \mathcal{P}_{\Cl_{\tau}}(X)&:=&\{A\subset X\colon A\text{ is nonempty and $\tau$-closed}\}.
\end{eqnarray*}
Now, let $X$ be a pseudometric space, and let $\tau$ be another topology on $X$. Define
\begin{eqnarray*}
  \mathcal{P}_{\Cl_{\tau}, \B}(X)&:=&\{A\subset X\colon A\text{ is nonempty, $\tau$-closed and bounded}\}.
\end{eqnarray*}
Let $X$ be a linear topological space. Define
\begin{eqnarray*}
  \mathcal{P}_{\Cl, \Conv}(X)&:=&\{A\subset X\colon A\text{ is nonempty, closed and convex}\}.
\end{eqnarray*}
Let $X$ be a linear topological space, and let $\tau$ be another topology on $X$. Define
\begin{eqnarray*}
  \mathcal{P}_{\Cl_{\tau}, \Conv}(X)&:=&\{A\subset X\colon A\text{ is nonempty, $\tau$-closed and convex}\}.
\end{eqnarray*}
Let $X$ be a normed space. Define
\begin{eqnarray*}
  \mathcal{P}_{\Cl, \B, \Conv}(X)&:=&\{A\subset X\colon A\text{ is nonempty, closed, bounded and convex}\}.
\end{eqnarray*}
Let $X$ be a normed space, and $\tau$ be another topology on $X$. Define
\begin{eqnarray*}
  \mathcal{P}_{\Cl_{\tau}, \B, \Conv}(X)&:=&\{A\subset X\colon A\text{ is nonempty, $\tau$-closed, bounded and convex}\}.
\end{eqnarray*}

We write $\mathcal{P}_{*}(X)$ to denote any of the above families of nonempty subsets of $X$, where the subscript ``$*$'' stands for one of the listed combinations of properties (closedness, boundedness, convexity, etc.).

\begin{definition}\label{dfn:hyperspace}
If one of the families $\mathcal{P}_{*}(X)$ is endowed with the Hausdorff distance induced by the pseudometric on $X$, then this family is called a hyperspace over $X$.
\end{definition}

The following lemma will be useful for us later.

\begin{lemma}\label{lem:ball}
Let $X$ be a pseudometric space and all elements of a hyperspace $\mathcal{P}_{*}(X)$ be bounded subsets of $X$. Then, for each closed ball $\mathbb{B}\subset \mathcal{P}_{*}(X)$ centered at an element of $\mathcal{P}_{*}(X)$, there exists a closed ball $B\subset X$ centered at a point of $X$ such that $\mathbb{B}$ coincides with a closed ball in $\mathcal{P}_{*}(B).$
\end{lemma}

\begin{proof}

Since $\mathbb{B}$ is a closed ball in $\mathcal{P}_{*}(X)$ and all elements of a hyperspace $\mathcal{P}_{*}(X)$ are bounded subsets, there exists an element $M\in \mathbb{B}$ and a radius $0\le r < \infty$ such that $\mathbb{B} = \bigl\{N\in \mathcal{P}_{*}(X)\colon d_H(M, N) \le r\bigr\}$. Since $M$ is bounded in $X$, there is a point $m\in M$ and a number $0\le R < \infty$ such that $M\subset B_R(m)$. By Lemma~\ref{sum_0}, for any $N\in \mathbb{B}$, we obtain $N\subset B_r(M) \subset B_r\bigl(B_R(m)\bigr) \subset B_{r+R}(m)\subset X.$ 

Denote $B_{r+R}(m)$ by $B$. Thus, $\mathbb{B}\subset \mathcal{P}_{*}(B)$. Hence, 
$$\mathbb{B} = \bigl\{N\in \mathcal{P}_{*}(B)\colon d_H(M, N) \le r\bigr\}.$$
The lemma is proved.\qed

\end{proof}

%%%%%%%%%%%%%%%%%%%%%%%%%%%%%%
\subsection{Vietoris Topology}\label{sec:Vietoris}
%%%%%%%%%%%%%%%%%%%%%%%%%%%%%%

\begin{definition}
Let $(X, \tau)$ be a topological space. The Vietoris topology is a topology on the set $\mathcal{P}_{\Cl_{\tau}}(X)$ defined by a subbase consisting of two types of sets for each $\tau$-open $U \subset X:$
\begin{itemize}
\item $U^- = \bigl\{M \in \mathcal{P}_{\Cl_{\tau}}(X) \colon M \cap U \neq \emptyset\bigr\};$
\item $U^+ = \bigl\{M \in \mathcal{P}_{\Cl_{\tau}}(X) \colon M \subset U\bigr\}.$
\end{itemize}
Thus, the subbase of the Vietoris topology is$:$
\begin{equation}\label{eq:subbase}
\mathcal{S} = \{U^- \colon U \in \tau\} \cup \{U^+ \colon U \in \tau\}.
\end{equation}
The Vietoris topology on $\mathcal{P}_{\Cl_{\tau}}(X)$ generated by $\tau$ will be denoted in this text by $V_{\tau}$.
\end{definition}

Note that any subset of a topological space is itself a topological space with the induced topology. Therefore, if $\mathcal{P}_{*}(X)\subset \mathcal{P}_{\Cl_{\tau}}(X)$, then we will assume that the Vietoris topology $V_{\tau}$ is also introduced on $\mathcal{P}_{*}(X)$, meaning the topology induced from $\mathcal{P}_{\Cl_{\tau}}(X)$. Indeed, the Vietoris topology $V_{\tau}$ can be introduced on the space $\mathcal{P}_{*}(X)$ by the same formula via the subbase~\eqref{eq:subbase}.

Analogously, the Vietoris topology can be introduced also on the space $\mathcal{P}_{\Cl_{\tau}}(X)\cup\{\emptyset\}$. Strickland proved in~\cite[Theorem~21.48]{Strickland} that if $(X,\tau)$ is a compact Hausdorff space, then $\mathcal{P}_{\Cl_{\tau}}(X)\cup\{\emptyset\}$ is compact Hausdorff in the Vietoris topology. Hence, $\mathcal{P}_{\Cl_{\tau}}(X)$ is also a Hausdorff space. Further, note that $\emptyset^+ = \{\emptyset\}$, and therefore $\{\emptyset\}$ is open with respect to the Vietoris topology. Hence, $\mathcal{P}_{\Cl_{\tau}}(X)$ is compact as a closed subset of a compact space. Thus, we have the following theorem.

\begin{theorem}[{\cite[Theorem~21.48]{Strickland}}]\label{thm:comp_hyp}
Let $(X, \tau)$ be a compact Hausdorff topological space. Then $\bigl(\mathcal{P}_{\Cl_{\tau}}(X), V_{\tau}\bigr)$ is also a compact Hausdorff topological space.
\end{theorem}

\begin{definition}
Let $(X, \tau)$ be a Hausdorff topological space and let $\{A_{\lambda}\}_{\lambda\in \Lambda}$ be a net in $\mathcal{P}_{\Cl_{\tau}}(X)$.
\begin{itemize}
\item The lower limit of the net $\{A_{\lambda}\}_{\lambda\in \Lambda}$ is the set of all points $x\in X$ such that for each neighborhood $U_x\in \tau$ of $x$, there exists an index $\lambda_0\in \Lambda$ such that for all $\lambda\ge \lambda_0$, we have $A_{\lambda}\cap U_x\neq\emptyset$. The lower limit of the net $\{A_{\lambda}\}_{\lambda\in \Lambda}$ is denoted by $\Li A_{\lambda}$.
\item The upper limit of the net $\{A_{\lambda}\}_{\lambda\in \Lambda}$ is the set of all points $x\in X$ such that for each neighborhood $U_x\in \tau$ of $x$ and for any index $\lambda_0\in \Lambda$, there exists $\lambda\ge \lambda_0$ such that $A_{\lambda}\cap U_x\neq\emptyset$. The upper limit of the net $\{A_{\lambda}\}_{\lambda\in \Lambda}$ is denoted by $\Ls A_{\lambda}$.
\item The net $\{A_{\lambda}\}_{\lambda\in \Lambda}$ is said to converge in the sense of Kuratowski to a subset $A\subset X$ if $\Li A_{\lambda} = \Ls A_{\lambda} = A$. In this case, the set $A$ is called the Kuratowski limit and is denoted by $\Lim\limits_{\lambda} A_{\lambda}$.
\end{itemize}
\end{definition}

The following proposition is well known, e.g., see~\cite{Beer}.

\begin{proposition}[{\cite[p.~145]{Beer}}]
The subsets $\Li A_{\lambda}$ and $\Ls A_{\lambda}$ are $\tau$-closed subsets of a Hausdorff topological space $(X, \tau)$, whether or not the terms $A_{\lambda}$ of the net are $\tau$-closed.
\end{proposition}

The following result in the book~\cite{Beer} is formulated for the Fell topology on $\mathcal{P}_{\Cl_{\tau}}(X)$. This topology differs from the Vietoris topology only in that, for forming the subbase, the sets $U^+$ are considered only for those $\tau$-open $U\subset X$ for which $X\setminus U$ is $\tau$-compact. Consequently, if the space $(X, \tau)$ is compact, then the Vietoris topology and the Fell topology on $\mathcal{P}_{\Cl_{\tau}}(X)$ coincide.

\begin{theorem}[{\cite[Theorem~5.2.6]{Beer}}]\label{thm:Kuratow}
Let $(X, \tau)$ be a compact Hausdorff topological space. A net $\{A_{\lambda}\}_{\lambda\in \Lambda}\subset \mathcal{P}_{\Cl_{\tau}}(X)$ converges in the topology $V_{\tau}$ to $A\in \mathcal{P}_{\Cl_{\tau}}(X)$ if and only if $\Lim\limits_{\lambda} A_{\lambda} = A$.
\end{theorem}

%%%%%%%%%%%%%%%%%%%%%%%%%%%%%%
\subsection{Weak and Weak$^*$ Topologies}\label{sec:weak_top}
%%%%%%%%%%%%%%%%%%%%%%%%%%%%%%

The weakest topology on a set $X$ with respect to which every functional $f\in \mathcal{F}$ remains continuous is denoted in functional analysis by $\sigma(X, \mathcal{F})$. Such a topology always exists, and it is the topology generated by the subbase $\bigl\{f^{-1}(U)\colon f\in \mathcal{F}, U \text{ is open in }\mathbb{R}\bigr\}$.

\begin{definition}
A family of functionals $\mathcal{F}$ on a set $X$ is said to separate points if for any $x\neq y$ in $X$ there exists $f\in \mathcal{F}$ such that $f(x)\neq f(y)$.
\end{definition}

\begin{definition}
A linear topological space $X$ is called locally convex if it has a base of neighborhoods of zero consisting of convex sets. In such a case, the linear topology of $X$ is also called a locally convex topology.
\end{definition}

\begin{definition}
Let $X$ be a linear topological space. Then the linear space of all continuous linear functionals on $X$ is called the dual space of $X$ and is denoted by $X^*$.
\end{definition}

The next theorem will be useful for us later.

\begin{theorem}[{\cite[Theorem~3.10]{Rudin}}]\label{thm:loc_conv}
Let $X$ be a linear space and let a space of linear functionals $\mathcal{F}$ on $X$ separate points. Then $\bigl(X, \sigma(X, \mathcal{F})\bigr)$ is a locally convex Hausdorff topological space and its dual space is $\mathcal{F}$, i.e., $X^* = \mathcal{F}$.
\end{theorem}

\begin{definition}
Let $X$ be a linear topological space. Then the topology $\sigma(X, X^*)$ is called the weak topology.
\end{definition}

Let $X$ be a linear topological space. For each element $x\in X$, consider the \textit{evaluation functional}
$$e_x\colon X^*\rightarrow \mathbb{R};$$
$$e_x\colon f\mapsto f(x).$$

A mapping which assigns to each element $x\in X$ the evaluation functional $e_x$ is called the \textit{canonical mapping}. Thus, $X$ can be canonically identified with a linear space of linear functionals defined on $X^*$.

\begin{definition}
Let $X$ be a linear topological space. Then the topology $\sigma(X^*, X)$ is called the weak$^*$ topology.
\end{definition}

\begin{remark}\label{rk:reflex}
Clearly, $X$ separates points on $X^*$. First, Theorem~\ref{thm:loc_conv} implies $X^{**}\cong X$, provided that $X^{**}$ is a dual space of $\bigl(X^*, \sigma(X^*, X)\bigr)$. Second, Theorem~\ref{thm:loc_conv} implies the following corollary.
\end{remark}

\begin{corollary}\label{cor:Haus_loc}
Let $X$ be a linear topological space. Then the weak$^*$ topology on $X^*$ is Hausdorff and locally convex.
\end{corollary}

We now formulate the Banach--Alaoglu theorem, which is standard in functional analysis. It will play an important role in the proof of some of our results. It can be found, for example, in~\cite[Theorem~3.15]{Rudin}. We state it in a form convenient for us and in the particular case of normed spaces that we need.

Let $X$ be normed and let the usual operator norm be specified on $X^*$:
\begin{equation}\label{eq:norm}
||f|| = \sup\limits_{||x|| = 1} |f(x)|.
\end{equation}
Then by~\cite{Rudin}, the Banach--Alaoglu theorem asserts that the closed unit ball (in the norm metric) in $X^*$ is weak$^*$-compact (that is, compact with respect to the weak$^*$ topology). By Corollary~\ref{cor:Haus_loc}, the weak$^*$ topology is locally convex. In any linear topology, by definition, translations and multiplications by nonzero scalars are homeomorphisms. Consequently, in a linear topology, all closed balls of nonzero radius (centered at arbitrary points of the space) are homeomorphic to each other. Therefore, from the compactness of the closed unit ball in such a topology, it follows that all closed balls are compact (a closed ball of radius zero is a singleton and hence compact in any topology). Thus, the Banach--Alaoglu theorem can be formulated in the following form, which will be convenient for us.

\begin{theorem}[Banach--Alaoglu Theorem]\label{thm:Ban_Ala}
Let $X$ be a normed space and let the norm~\eqref{eq:norm} be specified on $X^*$. Then each closed ball in $X^*$ is weak$^*$-compact.
\end{theorem}

\textbf{Henceforth, whenever we speak of a dual (bidual) space, we shall assume that it is normed and its norm is the operator norm~\bfseries\eqref{eq:norm}, induced by the norm of its predual space.}

%%%%%%%%%%%%%%%%%%%%%%%%%%%%%%
\subsection{Reflexive Space}\label{sec:ref_space}
%%%%%%%%%%%%%%%%%%%%%%%%%%%%%%

Note that each evaluation functional $e_y$, where $y\in X$, is continuous on $X^*$ in the topology generated by the norm~\eqref{eq:norm}. Indeed, when $y = 0$, we have $e_y(f) = f(0) = 0$ for any $f\in X^*$. And when $y\neq 0$, from $||f - g|| = \sup\limits_{||x|| = 1} |f(x) - g(x)|\le \varepsilon/||y||$ it follows that $|e_y(f) - e_y(g)| = |f(y) - g(y)| = ||y||\cdot |f(y/||y||) - g(y/||y||)|\le \varepsilon$. Hence, the normed space $X$ is canonically embedded into $X^{**}$. The space $X^{**}$ is sometimes called the \textit{bidual space} of $X$.

\begin{definition}
A normed space $X$ that is canonically isomorphic to its bidual $X^{**}$ is called reflexive.
\end{definition}

To obtain the result below, we will need the following theorem, which is sometimes called Mazur's theorem. In the book~\cite{Rudin}, this theorem is stated for the case of a locally convex Hausdorff topological space $X$ and the weak topology on it. But we have already agreed above that when we speak about the dual space, we assume that its predual is normed. Therefore, we will formulate Mazur's theorem within the framework of normed spaces.

\begin{theorem}[{\cite[Theorem~3.12]{Rudin}}]\label{thm:mazur}
Let $X$ be a normed space, let $A$ be a convex subset of $X$ and $\tau = \sigma(X, X^*)$. Then $\Cl_{\tau} A = \Cl A$.
\end{theorem}

Theorem~\ref{thm:mazur} yields the following direct corollary.

\begin{corollary}\label{cor:equal_0}
Let $X$ be a normed space and $\tau = \sigma(X, X^*)$. Then 
$$\mathcal{P}_{\Cl_{\tau}, \Conv}(X) = \mathcal{P}_{\Cl, \Conv}(X).$$
\end{corollary}

If $X$ is reflexive, then the canonical mapping is an isometric isomorphism, so $X\cong X^{**}$. In this case, $X$ can be identified with the dual space of $X^*$, i.e., $X^*$ is the predual of $X$. Hence, in that case, the weak$^*$ topology on $X$ is $\sigma(X, X^*)$, i.e., the weak and the weak$^*$ topologies coincide on reflexive spaces. Thus, the following corollary holds.

\begin{corollary}\label{cor:equal}
Let $X$ be a reflexive space, $\tau$ be the weak$^*$ topology on $X$. Then
$$\mathcal{P}_{\Cl_{\tau}, \Conv}(X) = \mathcal{P}_{\Cl, \Conv}(X).$$
\end{corollary}

We emphasize that here we are dealing with the coincidence of sets. The Vietoris topology $V_{\tau}$ on $\mathcal{P}_{\Cl_{\tau}, \B, \Conv}(X)$ and the metric topology on $\mathcal{P}_{\Cl, \B, \Conv}(X)$ induced by $d_H$ do not coincide in general.

%%%%%%%%%%%%%%%%%%%%%%%%%%%%%%
\subsection{Lower Semicontinuous Function}\label{sec:semicont}
%%%%%%%%%%%%%%%%%%%%%%%%%%%%%%

\begin{definition}
Let $X$ be a topological space and $f\colon X\rightarrow \mathbb{R}\cup\{-\infty, \infty\}$. The function $f$ is called lower semicontinuous at a point $x\in X$ if for any net $\{x_{\alpha}\}_{\alpha\in \Lambda}\subset X$ converging to $x$, we have $$\liminf\limits_{\alpha} f(x_{\alpha}) \ge f(x).$$
Accordingly, $f$ is called lower semicontinuous on $X$ if it is lower semicontinuous at every point of $X$.
\end{definition}

\begin{proposition}\label{prop:cont_level_0}
Let $X$ be a topological space and $f\colon X\rightarrow \mathbb{R}$. The function $f$ is lower semicontinuous if and only if for every $c\in \mathbb{R}$ the set $\{x\in X\colon f(x) \le c\}$ is closed. 
\end{proposition}

\begin{proof}

Assume that for every $c\in \mathbb{R}$ the set $\{x\in X\colon f(x) \le c\}$ is closed. Fix arbitrary $x'\in X$ and choose $c'\in \mathbb{R}$ such that $f(x') > c'$. In such a case, $U_{c'} = \{x\in X\colon f(x) > c'\}$ is open. Then $x'\in U_{c'}$ and $U_{c'}$ is a neighborhood of $x'$. Let $\{x_{\alpha}\}_{\alpha \in \Lambda}\subset X$ be an arbitrary net converging to $x'$.

By definition, since $x_{\alpha}\rightarrow x'$, for any neighborhood $U$ of $x'$ there exists an index $\alpha'$ such that for all $\beta \ge \alpha'$ we have $x_{\beta}\in U$. In our case $U = U_{c'}$, so for all $\beta \ge \alpha'$ we have $x_{\beta}\in U_{c'}$, i.e., $f(x_{\beta}) > c'$. But then $\inf\limits_{\beta \ge \alpha'} f(x_{\beta}) \ge c'$ and 
$$\sup\limits_{\alpha} \inf\limits_{\beta \ge \alpha} f(x_{\beta}) = \liminf\limits_{\alpha} f(x_{\alpha}) \ge c'.$$ 
Since $c'$ is arbitrary with condition $f(x') > c'$, then we obtain $\liminf\limits_{\alpha} f(x_{\alpha}) \ge f(x')$. And since $x'\in X$ is arbitrary, $f$ is lower semicontinuous on $X$.

Now conversely, suppose $f$ is lower semicontinuous on $X$. We show that $F_c = \{x\in X\colon f(x) \le c\}$ is closed for every $c\in \mathbb{R}$. Fix arbitrary $c'\in \mathbb{R}$. Let $\{x_{\alpha}\}\subset F_{c'}$ be an arbitrary net converging to $x'\in X$. We need to show that $x'\in F_{c'}$, i.e., $f(x')\le c'$. 

By assumption $\liminf\limits_{\alpha} f(x_{\alpha}) \ge f(x')$. But since $x_{\alpha}\in F_{c'}$, we have $f(x_{\alpha})\le c'$ for all $\alpha$. Then $\sup\limits_{\alpha} \inf\limits_{\beta \ge \alpha} f(x_{\beta})\le c'$. Consequently, $f(x')\le \liminf\limits_{\alpha} f(x_{\alpha}) \le c'$. Therefore $x'\in F_{c'}$. Thus $F_{c'}$ contains all its limit points, i.e., it is closed.\qed

\end{proof}

\begin{proposition}\label{prop:semicont_for_dual}
Let $X$ be a normed space and let $\tau$ be a linear Hausdorff topology on $X$ with respect to which each closed ball of $X$ is compact. Then the original metric on $X$ is lower semicontinuous on $X\times X$ with respect to the product topology $\tau\times\tau$.
\end{proposition}

\begin{proof}

In a Hausdorff topological space, every compact set is closed. Hence, every closed ball $B_r(0)$ of the space $X$ is $\tau$-closed for all $r$. By the linearity of the topology $\tau$, the map $f(x, y) = x - y$ is continuous on $X\times X$ in the topology $\tau\times \tau$. Note that 
$$f^{-1}\bigl(B_r(0)\bigr) = \bigl\{(x, y)\in X\times X\colon ||x - y|| \le r\bigr\}.$$
Since $f$ is continuous and each $B_r(0)$ is $\tau$-closed, it follows that each set $\bigl\{(x, y)\in X\times X\colon ||x - y|| \le r\bigr\}$ is $(\tau\times\tau)$-closed. By Proposition~\ref{prop:cont_level_0}, it means that the metric on $X$ is lower semicontinuous on $X\times X$ with respect to the product topology $\tau\times\tau$.\qed

\end{proof}

\begin{proposition}\label{prop:Haus_top}
Let $X$ be a metric space, and let a topology $\tau$ be introduced on $X$ such that the metric of the space $X$ is lower semicontinuous on $X\times X$ in the topology $\tau\times \tau$. Then $\tau$ is a Hausdorff topology on $X$.
\end{proposition}

\begin{proof}

By Proposition~\ref{prop:cont_level_0}, the lower semicontinuity of the metric means that for any $c\in \mathbb{R}$ the set  
$\bigl\{(x, y)\in X\times X\colon |x\, y|\le c\bigr\}$ is closed in the topology $\tau\times \tau$.  
Let $c = 0$, then the following set is closed in this topology:  
$$\bigl\{(x, y)\in X\times X\colon |x\, y|\le 0\bigr\} = \bigl\{(x, y)\in X\times X\colon |x\, y| = 0\bigr\} = \bigl\{(x, x)\colon x\in X\bigr\}.$$  
Hence, $W = X\setminus \bigl\{(x, x)\colon x\in X\bigr\}$ is open in the topology $\tau\times \tau$.  
Therefore, for any two distinct points $x, y\in X$, there exist $\tau$-open sets $U$ and $V$ such that $(x, y)\in U\times V\subset W$.  
Consequently, $U\times V\cap \bigl\{(x, x)\colon x\in X\bigr\} = \emptyset$. This means that there is no point $z\in X$ such that $z\in U\cap V$. Hence, $U\cap V = \emptyset$.  
Thus, $\tau$ is a Hausdorff topology on $X$.\qed

\end{proof}

The following result holds.

\begin{theorem}[{\cite[Theorem~1]{Dixmier}}]\label{thm:Dixmier}
Let $X$ be a normed space such that there exists a Hausdorff locally convex topology $\tau$ on $X$ with respect to which the closed unit ball of $X$ is compact. Then $X$ is a dual space.
\end{theorem}

Thus, from Proposition~\ref{prop:semicont_for_dual}, Proposition~\ref{prop:Haus_top}, and Theorem~\ref{thm:Dixmier}, we obtain the following criterion for a dual space.

\begin{theorem}\label{thm:dual}
A normed space $X$ is a dual space if and only if there exists a locally convex topology $\tau$ on $X$ such that
\begin{itemize}
\item each closed ball of $X$ is $\tau$-compact;
\item the metric of $X$ is lower semicontinuous on $X\times X$ with respect to the product topology $\tau\times \tau$.
\end{itemize}
If $X$ is dual, then the weak$^*$ topology can be taken as such a topology $\tau$.
\end{theorem}

\begin{proof}

Let $X$ be a dual space. Introduce on $X$ the weak$^*$ topology $\tau$. By Theorem~\ref{thm:Ban_Ala}, each closed ball in $X$ is weak$^*$-compact. By Corollary~\ref{cor:Haus_loc}, the weak$^*$ topology is a Hausdorff locally convex linear topology. Therefore, by Proposition~\ref{prop:semicont_for_dual}, the metric of $X$ is lower semicontinuous on $X\times X$ with respect to the product topology $\tau\times \tau$.

Now conversely, let there exist a locally convex topology $\tau$ on $X$ such that
\begin{itemize}
\item[(1)] each closed ball of $X$ is $\tau$-compact;
\item[(2)] the metric of $X$ is lower semicontinuous on $X\times X$ with respect to the product topology $\tau\times \tau$.
\end{itemize}
By Proposition~\ref{prop:Haus_top}, the item (2) implies that $\tau$ is a Hausdorff topology. But then by Theorem~\ref{thm:Dixmier}, the space $X$ is a dual space.\qed

\end{proof}

%%%%%%%%%%%%%%%%%%%%%%%%%%%%%%
\subsection{Complemented Space}\label{sec:complemented}
%%%%%%%%%%%%%%%%%%%%%%%%%%%%%%

In this subsection, we provide some definitions and results from~\cite{Rudin} which will be useful for us later. Namely, in this subsection, we show for Banach spaces $X$ and $Y$ that the existence of a projector $P\colon Y\rightarrow X$ of norm $1$ is equivalent to the space $X$ being $1$-complemented in $Y$.

\begin{definition}
A metric on a linear space $X$ is called invariant if for all $x, y, z\in X$ one has $|x+z\,\, y+z| = |x\, y|$.
\end{definition}

\begin{definition}
Let $X$ be a Hausdorff topological linear space. If the topology of $X$ is induced by a complete invariant metric, then $X$ is called an $F$-space.
\end{definition}

\begin{definition}\label{dfn:comp}
Let $M$ be a closed subspace of a Hausdorff topological linear space $X$. If there exists a closed subspace $N$ of $X$ such that
$$X = M + N$$
and
$$M\cap N = \{0\},$$
then $M$ is said to be complemented in $X$. In this case one says that $X$ is the direct sum of $M$ and $N$, and the notation $X = M \oplus N$ is used.
\end{definition}

\begin{remark}
A complemented subspace in a complete metric space is complete due to closedness.
\end{remark}

\begin{definition}
Let $X$ be a linear space. A linear map $P\colon X\rightarrow X$ is called a projection $($or projector$)$ in $X$ if $P^2 = P$, i.e., if $P(Px)= Px$ for all $x\in X$.
\end{definition}

\begin{theorem}[{\cite[Theorem~5.16]{Rudin}}]\label{thm:rud}
\begin{itemize}
\item[ ]
\item[$($a$)$] If $P$ is a continuous projection in a Hausdorff linear topological space $X$, then $X = \Im P \oplus \Ker P.$
\item[$($b$)$] Conversely, if $X$ is an $F$-space and if $X = A \oplus B$, then the projection $P$ with $\Im P = A$ and $\Ker P = B$ is continuous.
\end{itemize}
\end{theorem}

\begin{corollary}\label{cor:rud}
A subspace of an $F$-space $X$ is complemented in $X$ if and only if it is the image of a continuous projection $P$ in $X$.
\end{corollary}

Thus, in the context of $F$-spaces, by Corollary~\ref{cor:rud}, the property of being complemented is equivalent to the existence of a continuous projection. This fact will be used later in the setting of normed spaces. Note that if an $F$-space is normed, then it is automatically a Banach space by definition. Thus, for Banach spaces, we will use the following equivalent definition.

\begin{definition}\label{dfn:lambda}
A Banach subspace of a Banach space $X$ is said to be complemented in $X$ if it is the image of a continuous projection $P$ in $X$. If, moreover, $||P||\le \lambda$, then this Banach subspace is said to be $\lambda$-complemented in $X$.
\end{definition}

Further, let $M$ be a complemented subspace in a Banach space $X$. Then by definition there exists a continuous projection $P$ in $X$ such that $\Im P = M$. Let $x\in \Im P = M$ and $x\neq 0$. There exists $y\in X$ such that $x = Py$. By the idempotent property we have:
$$x = Py = P(Py) = Px.$$
Hence, $Px = x$ for all $x\in M$. Moreover, note that
$$||x|| = ||Px|| \le ||P|| \cdot ||x||,$$
and therefore
$$1 \le ||P||.$$
Thus, in the case of Banach spaces, if $\lambda = 1$ in Definition~\ref{dfn:lambda}, then the norm of the projector $P$ becomes equal to $1$, and the corresponding Banach subspace is called $1$-complemented. So, we obtain the following proposition which will be useful for us later.

\begin{proposition}\label{prop:one_comp}
A Banach space $X$ is $1$-complemented in a Banach space $Y$ if and only if there exists a projector $P\colon Y\rightarrow X$ such that $||P|| = 1$.
\end{proposition}

%%%%%%%%%%%%%%%%%%%%%%%%%%%%%%
\subsection{Minimal Networks}
%%%%%%%%%%%%%%%%%%%%%%%%%%%%%%

In this subsection, we recall the necessary concepts from graph theory and fix the corresponding notation. More detailed information on graph theory can be found, for example, in~\cite{Graphs}.

\begin{definition}
A simple graph is a pair $(V, E)$ consisting of a finite set $V$ and a set $E$ of two-element subsets $\{u, v\} \subset V$. The elements of $V$ are called vertices of the graph, and the elements of $E$ are called edges of the graph. If $e = \{u, v\} \in E$, then the vertices $u$ and $v$ are said to be adjacent and are joined by the edge $e$.
\end{definition}

Since we will only consider simple graphs henceforth, the word ``simple'' will be omitted.

Henceforth, we will only consider undirected graphs, so we look at each edge $\{u, v\}\in E$ as an unordered pair of vertices, that is, we assume that $\{u, v\}$ and $\{v, u\}$ are the same edge.

We also emphasize that each edge is precisely a two-element set, that is, if $\{u, v\}$ is an edge of a graph, then $u\neq v.$

\begin{definition}
Let $G = (V, E)$ be a connected graph, and let $\mathcal{A} \subset V$. Then $G$ is said to join $\mathcal{A}$. The vertices in $\mathcal{A}$ are called boundary vertices, the set $\mathcal{A}$ itself is called the boundary of the graph, and the vertices in $V \setminus \mathcal{A}$ are called interior vertices.
\end{definition}

Thus, the boundary of a graph is some distinguished and fixed subset of its vertices. We will henceforth denote the boundary of the graph $G$ by $\partial G$.

\begin{definition} 
Let $G = (V, E)$ be a connected graph. A parametric network of type $G$ in a pseudometric space $X$ is an arbitrary mapping $g\colon V\rightarrow X$. The graph $G$ is called the parameterizing graph of $g$.
\end{definition}

Note that the mapping $g$ induces a corresponding mapping on the edges of the graph $G$ by the rule $g_e\colon \{v_i, v_j\} \mapsto \bigl\{g(v_i), g(v_j)\bigr\}$.

\begin{definition}
The elements of $g(V)$ are called vertices of the network, and the elements of $g_e(E)$ are called edges of the network.
\end{definition}

\textit{The length of an edge} $\bigl\{g(v_i), g(v_j)\bigr\}$ of the network $g$ is defined as the distance $\bigl|g(v_i)\, g(v_j)\bigr|$, and \textit{the length of the network} $g$ is defined as the sum of the lengths of all its edges:
$$|g| := \sum\limits_{\{v_i, v_j\}\in E} \bigl|g(v_i)\, g(v_j)\bigr|.$$

Let $\partial G$ be a subset of a pseudometric space $X$. We shall assume that in this case all networks $g$ satisfy the following condition: \textbf{the restriction of {\boldmath $g$} to {\boldmath $\partial G$} is the identity mapping}. Thus, each such network $g$ is uniquely determined by the images of the interior vertices of the parameterizing graph $G$.

\begin{definition}
Let a connected graph $G_1$ parameterizing a network $g_1$ and a connected graph $G_2$ parameterizing a network $g_2$ be given and $\mathcal{A} = \partial G_1 = \partial G_2$. These networks are said to be of the same type (say, $G_1$) if there exists an isomorphism between the graphs $G_1$ and $G_2$ that is the identity on $\mathcal{A}$.
\end{definition}

For a fixed boundary $\mathcal{A}$ in a pseudometric space $X$, consider a connected graph $G = (V, E)$ joining $\mathcal{A}$, that is $\partial G = \mathcal{A}\subset X$. Denote by $[G, \mathcal{A}]$ the set of all networks of this type $G$ in the pseudometric space $X$. By the convention made above, all networks $g\in [G, \mathcal{A}]$ satisfy the condition $g\big|_{\partial G = \mathcal{A}} = \text{id}$. 

\begin{definition}\label{dfn:mpn}
A network from $[G, \mathcal{A}]$ that has the smallest possible length among all networks in $[G, \mathcal{A}]$ is called a minimal parametric network of type $G$ joining $\mathcal{A}$.
\end{definition}

A minimal parametric network does not always exist. Nevertheless, the infimum of the quantities $|g|$ over all networks $g\in [G, \mathcal{A}]$ always exists. 

\begin{definition}
The quantity $$\mpn[G, \mathcal{A}] = \inf\limits_{g\in [G, \mathcal{A}]} |g|$$ is called the \textit{length of the minimal parametric network.}
\end{definition}

\begin{definition}
\textit{The length of the minimal Steiner tree} joining a subset $\mathcal{A}$ of a pseudometric space $X$ is the value
$$\smt(\mathcal{A}) = \inf\limits_{G\colon \partial G = \mathcal{A}} \mpn[G, \mathcal{A}].$$
\end{definition}

The number of elements in a set $M$ is denoted by $\# M$.

\begin{definition}\label{dfn:mst}
Let $X$ be a pseudometric space, let a graph $G = (V, E)$ join $\mathcal{A}\subset X$, let $g\in [G, \mathcal{A}]$ and $|g| = \smt(\mathcal{A})$. If there does not exist a graph $G' = (V', E')$ joining $\mathcal{A}$ and a network $g'\in [G', \mathcal{A}]$ such that $|g'| = \smt(\mathcal{A})$ and $\# (V'\setminus \mathcal{A}) < \# (V\setminus \mathcal{A})$, then $g$ is called a \textit{minimal Steiner tree joining} $\mathcal{A}$.
\end{definition}

In the case of a metric space, the following two propositions can be found, for example, in~\cite{Malaysia}, Proposition~9 and~10. However, there they are proved for the case when $X$ is a metric space. But the proofs for the case of pseudometric spaces carry over verbatim, so we omit them.

\begin{proposition}
Let $X$ be a pseudometric space, a graph $G$ join $\mathcal{A}\subset X$ and $g\in [G, \mathcal{A}]$ be a minimal Steiner tree joining $\mathcal{A}$. Then any interior vertex of $G$ has at least $3$ adjacent vertices.
\end{proposition}

\begin{proposition}\label{prop:int_vert}
Let $X$ be a pseudometric space, a graph $G = (V, E)$ join $\mathcal{A}\subset X$, $\# \mathcal{A} = n$ and $g\in [G, \mathcal{A}]$ be a minimal Steiner tree joining $\mathcal{A}$. Then the graph $G$ has at most $n-2$ interior vertices.
\end{proposition}

\begin{remark}\label{rk:finite_set}
By Proposition~\ref{prop:int_vert}, the number of interior vertices of a minimal Steiner tree is at most $n-2$. For each fixed $m \le n-2$, there are only finitely many connected graphs with $n+m$ vertices, where $n$ vertices are distinguished and fixed as the boundary set $\mathcal{A}$. Therefore, up to isomorphism preserving the boundary, the family of all graphs $G$ with $\partial G = \mathcal{A}$ is finite. Hence, the infimum in the definition of $\smt(\mathcal{A})$ can be replaced by a minimum:
$$\smt(\mathcal{A}) = \min_{G\colon \partial G = \mathcal{A}} \mpn[G, \mathcal{A}].$$
\end{remark}

\section{Main Part: Theorems on Existence of Minimal Parametric Networks}\label{Main}

\textbf{From now on, by a closed ball in a pseudometric space we will mean a closed ball in the sense of Definition~\ref{dfn:balls} centered at a point of this space.}

%%%%%%%%%%%%%%%%%%%%%%%%%%%%%%
\subsection{Existence Theorems in Pseudometric Spaces}\label{sec:pseudo}
%%%%%%%%%%%%%%%%%%%%%%%%%%%%%%

\begin{definition}
\textit{A proper metric space} is a metric space in which every closed bounded subset is compact.
\end{definition}

In the book~\cite{IT}, the existence of a minimal parametric network, as well as of a minimal Steiner tree, on a connected smooth complete Riemannian manifold is proved; see~\cite[Chapter~2, Paragraph~2, Section~2.1, Theorem~2.1]{IT} and~\cite[Chapter~2, Paragraph~2, Section~2.1, Corollary~2.1]{IT}. As observed by Tropin in~\cite{Tropin}, the proof of Theorem~2.1 (and likewise of Corollary~2.1) carries over verbatim to the broader case of proper metric spaces (by the Hopf--Rinow theorem, a connected smooth Riemannian manifold is proper precisely when it is complete).

In the book~\cite{Ambrosio}, Theorem~4.5.9 is proved. In terms of geometric measure theory, it states that if in a metric space $(X, d)$ every closed ball $B$ can be endowed with some compact metrizable topology $\tau$ such that $d$ is $\tau$-lower semicontinuous on $B\times B$, then for every compact subset $S\subset X$ there exists a connected closed set $C\supset S$ of minimal one-dimensional Hausdorff measure. Note that a proper metric space is a special case of a space $X$ satisfying these conditions.

However, Theorem~4.5.9 does not directly extend the result of~\cite[Theorem~1]{Tropin}, even if we take $S$ to be a finite set, i.e., the boundary of the graph parameterizing the network. Indeed, first, in Theorem~4.5.9 the minimum is sought over a wide class of sets, which may not have the structure of a finite metric graph; for instance, they may exhibit fractal-like sets with infinite branching. Second, in~\cite[Theorem~1]{Tropin}, what is sought is not a connected subset, but rather a network, i.e., a mapping from the vertices of a finite graph into a metric space. Third, it is unknown whether the existence of a minimal Steiner tree for every finite collection implies the existence of a minimal parametric network of an arbitrary prescribed type for that collection.

In what follows, we extend Theorem~4.5.9 to the setting of parametric networks in pseudometric spaces, without requiring a compact topology on each closed ball to be metrizable. For convenience, we introduce the following definition.

\begin{definition}\label{dfn:equipped}
A pseudometric space $X$ is called \textit{compactly equipped} if every closed ball $B$ in $X$ can be equipped with its own topology $\tau$ satisfying$:$
\begin{itemize}
\item $(B, \tau)$ is a compact topological space$;$
\item the pseudometric of $X$ is lower semicontinuous on $B\times B$ with respect to the product topology $\tau\times \tau.$
\end{itemize}
\end{definition}

\begin{theorem}\label{thm:Ambrosio}
Let $X$ be a compactly equipped space. Then, for any finite set $\mathcal{A}\subset X$ and for any connected graph $G = (V, E)$ with boundary $\mathcal{A}$, there exists a minimal parametric network of type $G$ joining $\mathcal{A}$.
\end{theorem}

\begin{proof}

Let $m = \mpn[G, \mathcal{A}]$. For each $n\in \mathbb{N}$ choose $g_n\in [G, \mathcal{A}]$ such that $|g_n| < m + \frac{1}{n}$. Then the sequence of lengths $|g_n|$ is bounded by $m+1$. 

Fix a point $a\in \mathcal{A}$. Since the graph $G = (V, E)$ is connected, for any vertex $v\in V$ there exists a path $\{v_1 = v, v_2\}, \{v_2, v_3\},\ldots , \{v_{k-1}, v_k = a\}$ connecting $v$ to $a$. By the triangle inequality we get
$$
\bigl|g_n(v)\, a\bigr| \le \sum_{i=1}^{k-1} \bigl|g_n(v_i)\, g_n(v_{i+1})\bigr| \le |g_n| \le m+1
$$
for all $n$. Hence, all points $g_n(v)$ lie in the closed ball $B = B_{m+1}(a)$. By the definition of a compactly equipped space, on $B$ there exists a compact topology $\tau$ such that the pseudometric of the space $X$ is lower semicontinuous on $B\times B$ with respect to the topology $\tau\times \tau$.

Let $V_{\text{int}} = V\setminus \mathcal{A}$ be the set of interior vertices. Any network $g\in [G, \mathcal{A}]$ with images in $B$ is uniquely determined by the restriction $g|_{V_{\text{int}}}$, since $g|_{\mathcal{A}} = \mathrm{id}$. Consider the space
$$
Y = \prod_{v\in V_{\text{int}}} B,
$$
endowed with the product topology $\tau_{Y}$ generated by $\tau$. By Tychonoff's theorem, $Y$ is compact.

Consider the sequence of image tuples $\bigl\{g_n(V_{\text{int}})\bigr\}\subset Y$. It is well known that in a compact space every net (in particular, every sequence) has a convergent subnet. Let $\bigl\{g_{\alpha}(V_{\text{int}})\bigr\}_{\alpha\in \Lambda}$ (provided by $g_{\alpha}(V_{\text{int}}) = g_{n(\alpha)}(V_{\text{int}})$) be a subnet converging in $\tau_Y$ to some point $\widetilde{g}(V_{\text{int}})\in Y$. This means that for each vertex $v\in V_{\text{int}}$, the net $\bigl\{g_{\alpha}(v)\bigr\}\subset B$ converges to $\widetilde{g}(v)\in B$ in the topology $\tau$. For boundary vertices $a\in \mathcal{A}$, set $\widetilde{g}(a) = a$. Thus, $\widetilde{g}\in [G, \mathcal{A}]$.

Estimate the length of $\widetilde{g}$. For any edge $\{u, v\}\in E$, we have $g_{\alpha}(u)\rightarrow \widetilde{g}(u)$ and $g_{\alpha}(v)\rightarrow \widetilde{g}(v)$ in the topology $\tau$. By the lower semicontinuity of the pseudometric of $X$ in the topology $\tau\times \tau$, we have
$$
\bigl|\widetilde{g}(u)\, \widetilde{g}(v)\bigr| \le \liminf_{\alpha} \bigl|g_{\alpha}(u)\, g_{\alpha}(v)\bigr|.
$$
Summing over all edges, we obtain
\begin{multline*}
|\widetilde{g}| = \sum_{\{u, v\}\in E} \bigl|\widetilde{g}(u)\, \widetilde{g}(v)\bigr| \le \sum_{\{u, v\}\in E} \liminf_{\alpha} \bigl|g_{\alpha}(u)\, g_{\alpha}(v)\bigr| \le \\ \le \liminf_{\alpha} \sum_{\{u, v\}\in E} \bigl|g_{\alpha}(u)\, g_{\alpha}(v)\bigr| = \liminf_{\alpha} |g_{\alpha}|.
\end{multline*}
Since $\bigl\{g_{\alpha}\bigr\}_{\alpha\in \Lambda}$ is a subnet of the sequence $\{g_n\}$, for which $\lim\limits_{n\rightarrow \infty} |g_n| = m$, we have $\lim\limits_{\alpha} |g_{\alpha}| = m$. Consequently, $\liminf\limits_{\alpha} |g_{\alpha}| = m$, whence
$$
\mpn[G, \mathcal{A}]\le |\widetilde{g}| \le m = \mpn[G, \mathcal{A}].
$$

Thus, $|\widetilde{g}| = \mpn[G, \mathcal{A}]$, which by definition means that $\widetilde{g}$ is a minimal parametric network of type $G$ joining $\mathcal{A}$.\qed

\end{proof}

From Theorem~\ref{thm:Ambrosio} and Remark~\ref{rk:finite_set} we obtain

\begin{corollary}\label{cor:Ambrosio}
Let $X$ be a compactly equipped space. Then, for any finite set $\mathcal{A}\subset X$, there exists a minimal Steiner tree joining $\mathcal{A}$.
\end{corollary}

Indeed, compactly equipped spaces form a very broad class of spaces. Proper metric spaces belong to this class: every closed ball is compact already in the original topology, and the metric of the space is jointly continuous in both variables in any metric space. Thus, Theorem~\ref{thm:Ambrosio} and Corollary~\ref{cor:Ambrosio} generalize the result of~\cite[Theorem~1]{Tropin} and also the result of~\cite[Chapter~2, Paragraph~2, Section~2.1, Theorem~2.1]{IT}.

Note that, from Theorem~\ref{thm:dual}, we have the following corollary.

\begin{corollary}\label{cor:dual_equip}
A dual space is compactly equipped.
\end{corollary} 

Indeed, the following proposition also holds.

\begin{proposition}\label{prop:complete}
A compactly equipped space is complete.
\end{proposition}

\begin{proof}

Let $X$ be a compactly equipped pseudometric space. Consider an arbitrary Cauchy sequence $\{x_n\}_{n=1}^{\infty}\subset X$. By the definition of a Cauchy sequence, there exists a natural number $N$ such that for all $n, m \ge N$ we have $|x_n\, x_m| < 1$. In particular, for all $n\ge N$ we have $|x_n\, x_N| < 1$. Put
$$r = \max\bigl\{1, |x_1\, x_N|, |x_2\, x_N|, \ldots, |x_{N-1}\, x_N| \bigr\}.$$
Then $\{x_n\}_{n=1}^{\infty}\subset B_r(x_N)$. Denote $B_r(x_N)$ by $B$.

By the definition of a compactly equipped space, on the ball $B$ there exists a compact topology $\tau$ such that the pseudometric of the space $X$ is lower semicontinuous on $B\times B$ with respect to the topology $\tau\times \tau$. In a compact topological space, every sequence has a convergent subnet. Hence, there exists a subnet $\{x_{\alpha}\}_{\alpha\in \Lambda}$ (where $x_{\alpha} = x_{n(\alpha)}$) converging in the topology $\tau$ to some point $x\in B$.

Fix $\varepsilon > 0$. From the Cauchy property, there exists $L$ such that $|x_i\, x_j| < \varepsilon$ for all $i, j\ge L$. Moreover, for any fixed $i\ge L$, the function $f_i(y) = |x_i\, y|$ is lower semicontinuous on $(B, \tau)$. Therefore,
\begin{equation}\label{eq:i_alpha}
|x_i\, x| = f_i(x) \le \liminf\limits_{\alpha} f_i(x_{\alpha}) = \liminf\limits_{\alpha} |x_i\, x_{\alpha}|.
\end{equation}

By the definition of subnet, for $L$ there exists $\alpha_0\in \Lambda$ such that for all $\alpha\ge \alpha_0$ we obtain $n(\alpha)\ge L$. This means that for all such $\alpha\ge \alpha_0$ we have $|x_i\, x_{\alpha}| < \varepsilon$. Consequently, since $\Lambda$ is a direct set, we obtain $\liminf\limits_{\alpha} |x_i\, x_{\alpha}|\le \varepsilon$. Thus, by~\eqref{eq:i_alpha} for all $i\ge L$,
$$|x_i\, x|\le \varepsilon.$$
Since $\varepsilon > 0$ was arbitrary, it follows that $|x_i\, x|\rightarrow 0$ as $i\rightarrow \infty$. Hence, the sequence $\{x_n\}_{n=1}^{\infty}\subset X$ converges to $x\in X$ in the pseudometric of $X$. Therefore, $X$ is complete.\qed

\end{proof}

Note that there also exist incomplete metric spaces such that every finite subset can be joined by a minimal parametric network of arbitrary type. We shall formulate an example of such a space as a proposition.

\begin{proposition}\label{prop:noncomplete}
Let $X = \mathbb{Q}$ be the set of rational numbers with metric $|x\, y| = |x - y|$ for all $x, y\in \mathbb{Q}$. Then for any finite set $\mathcal{A}\subset \mathbb{Q}$ and any connected graph $G=(V, E)$ with boundary $\mathcal{A}$, there exists a minimal parametric network of type $G$ joining $\mathcal{A}$.
\end{proposition}

\begin{proof}

For each interior vertex $v\in V\setminus \mathcal{A}$, we introduce a variable $x_v\in \mathbb{R}$. For each edge $e = \{u, v\}\in E$, we introduce a variable $t_e\in \mathbb{R}$. Consider the following linear programming problem:
\begin{equation}\label{eq:prog}
\sum\limits_{e\in E} t_e \rightarrow \min
\end{equation}
subject to the constraints for all $\{u, v\}\in E$
$$t_{\{u, v\}}\ge x_u - x_v;$$
$$t_{\{u, v\}}\ge x_v - x_u$$
and for all $a\in \mathcal{A}$
$$x_a = a.$$
All constraints are linear, and therefore this problem can indeed be reduced to the standard linear programming form $Ax = b$, $x\ge 0$. The fundamental theorem of linear programming guarantees that if the set $P = \{x\colon Ax = b, x\ge 0\}$ is nonempty and the objective function is bounded below on $P$, then there is an optimal basic feasible solution (i.e., an optimal basic solution of $Ax = b$ that satisfies $x\ge 0$). 

The constraints imply $t_e\ge 0$ for all $e\in E$, therefore $\sum_{e\in E} t_e\ge 0$. Also one can take $x_v = 0$ for all $v\in V\setminus \mathcal{A}$ and put $t_e = |x_u - x_v|$ for all $\{u, v\}\in E$. Such values of variables satisfy all constraints, so $P\neq \emptyset$. Hence, there is an optimal basic feasible solution for our task~\eqref{eq:prog}.

Since all coefficients in our constraint system are rational, by Cramer's rule any basic solution of this system is rational (the determinant of a rational matrix is a rational number). Hence, there exists an optimal solution $(x^*, t^*)$ with $x_v^* \in \mathbb{Q}$ for all $v\in V\setminus \mathcal{A}$ and $t_{\{u, v\}}^* \in \mathbb{Q}$.

For every edge $e = \{u, v\}$, in an optimal solution we must have $t_{\{u, v\}}^* = |x_u^* - x_v^*|$; otherwise, we could decrease the corresponding $t_{\{u, v\}}^*$ without violating the constraints and improve the objective function. Now define a network $g\colon V\to \mathbb{Q}$ by setting $g(v) = x_v^*$ for $v\in V\setminus \mathcal{A}$ and $g(a)=a$ for $a\in \mathcal{A}$. Then
$$
|g| = \sum_{\{u, v\}\in E} |g(u) - g(v)| = \sum_{e\in E} t_e^*,
$$
and this value equals $\mpn[G, \mathcal{A}]$. So, $g$ is a minimal parametric network of type $G$ joining $\mathcal{A}$.\qed

\end{proof}

%%%%%%%%%%%%%%%%%%%%%%%%%%%%%%
\subsection{Existence Theorems in Banach Spaces}\label{sec:Banach}
%%%%%%%%%%%%%%%%%%%%%%%%%%%%%%

In the article~\cite{Bednov}, Strelkova gave an example of a Banach space (Theorem~3) in which for any $n\ge 3$ there exists a collection $\mathcal{A} = \{x_1,\ldots , x_n\}$ for which no minimal Steiner tree joining this collection exists. In the same article, Bednov also gave a sufficient condition on a Banach space $X$ (Theorem~1) for any finite collection in $X$ to be joinable by a shortest network, i.e., a minimal Steiner tree. This condition is precisely the existence of a norm-one projector $P\colon X^{**}\rightarrow X$. Note that in fact, Bednov's proof establishes the existence of a minimal network for each fixed type; at the end of the proof, he then invokes the finiteness of the number of possible types to conclude that a shortest network exists.

So, a bidual is a dual space, which means that $X^{**}$ is a compactly equipped space according to Corollary~\ref{cor:dual_equip}, and so, by Theorem~\ref{thm:Ambrosio}, any finite set in it can be joined by a minimal parametric network and hence by a minimal Steiner tree. Finally, the projector of norm $1$ on $X^{**}$ allows us to obtain a minimal network in $X$ itself. Thus, we can replace $X^{**}$ by an arbitrary compactly equipped normed space $Y$. The main idea of the proof remains unchanged, but the argument becomes simpler.

Note that, by Proposition~\ref{prop:complete}, a compactly equipped normed space is a Banach space. Also according to item ($a$) of Theorem~\ref{thm:rud} and Definition~\ref{dfn:comp}, the image of $P$ (which is equal to $X$) is closed in $Y$. Hence, $X$ is a Banach space too. According to Proposition~\ref{prop:one_comp}, for Banach spaces $X$ and $Y$, the existence of a projector $P\colon Y\rightarrow X$ of norm $1$ is equivalent to $X$ being $1$-complemented in $Y$. Thus we obtain the following theorem.

\begin{theorem}\label{thm:gen}
Let $X$ be a Banach space that is $1$-complemented in some compactly equipped Banach space $Y$ $($in particular, in its bidual$)$. Then, for any finite set $\mathcal{A}\subset X$ and for any connected graph $G = (V, E)$ with boundary $\mathcal{A}$, there exists a minimal parametric network of type $G$ joining $\mathcal{A}$. 
\end{theorem}

\begin{proof}

Since $X$ is a complemented space in $Y$, we have $\mathcal{A}\subset Y$. By Theorem~\ref{thm:Ambrosio}, in $Y$ there exists a minimal parametric network $\widetilde{g}$ of type $G$ joining $\mathcal{A}$. 

Note that $|\widetilde{g}| = \mpn\bigl[G, \mathcal{A}\subset Y\bigr]\le \mpn\bigl[G, \mathcal{A}\subset X\bigr]$, since $X\subset Y$, and, for a vertex $v\in V\setminus \mathcal{A}$, an image $\widetilde{g}(v)$ can lie in $Y\setminus X$. But by definition of $1$-complemented space, there exists a projector $P$ of norm $1$ such that $P(Y) = X$.

Apply the projector $P$ to the image $\widetilde{g}(v)\in Y$ of each interior vertex $v\in V\setminus \mathcal{A}$ by setting $g(v) = P\bigl(\widetilde{g}(v)\bigr)\in X$. This defines the network $g$ of type $G$ (identical on $\mathcal{A}$) in space $X$ joining $\mathcal{A}\subset X$. Then for any edge $\{u, v\}\in E$ we obtain:
$$
\bigl|g(u)\, g(v)\bigr| = \Bigl|P\bigl(\widetilde{g}(u)\bigr)\, P\bigl(\widetilde{g}(v)\bigr)\Bigr| \le ||P||\cdot \bigl|\widetilde{g}(u)\, \widetilde{g}(v)\bigr| = \bigl|\widetilde{g}(u)\, \widetilde{g}(v)\bigr|.
$$
Consequently,
$$
\mpn\bigl[G, \mathcal{A}\subset X\bigr] \le |g| \le |\widetilde{g}| \le \mpn\bigl[G, \mathcal{A}\subset X\bigr].
$$
Thus, $|g| = \mpn\bigl[G, \mathcal{A}\subset X\bigr]$. Hence, $g$ is a minimal parametric network of type $G$ joining $\mathcal{A}$ in $X$.\qed

\end{proof}

\begin{corollary}\label{cor:gen}
Let $X$ be a Banach space that is $1$-complemented in some compactly equipped Banach space $Y$ $($in particular, in its bidual$)$. Then, for any finite set $\mathcal{A}\subset X$, there exists a minimal Steiner tree joining $\mathcal{A}$. 
\end{corollary}

%%%%%%%%%%%%%%%%%%%%%%%%%%%%%%
\subsection{Existence Theorems in Hyperspaces Over Pseudometric Spaces}\label{sec:hyp_pseudo}
%%%%%%%%%%%%%%%%%%%%%%%%%%%%%%

\begin{theorem}\label{thm:dH_semicont}
Let $X$ be a pseudometric space and let a closed ball $B$ in $X$ be equipped with its own topology $\tau$ such that$:$
\begin{itemize}
\item $(B, \tau)$ is a compact Hausdorff topological space$;$
\item the pseudometric of $X$ is lower semicontinuous on $B\times B$ with respect to the product topology $\tau\times \tau$.
\end{itemize}
Then the Hausdorff distance $d_H$ induced by the pseudometric of $X$ is lower semicontinuous on $\mathcal{P}_{\Cl_{\tau}}(B)\times \mathcal{P}_{\Cl_{\tau}}(B)$ in the topology $V_{\tau}\times V_{\tau}$.
\end{theorem}

\begin{proof}

Let $\bigl\{(M_{\alpha}, N_{\alpha})\bigr\}_{\alpha\in \Lambda}$ be a net in $\mathcal{P}_{\Cl_{\tau}}(B)\times \mathcal{P}_{\Cl_{\tau}}(B)$ converging to $(M, N)\in \mathcal{P}_{\Cl_{\tau}}(B)\times \mathcal{P}_{\Cl_{\tau}}(B)$ in the topology $V_{\tau}\times V_{\tau}$. We will show that $\liminf\limits_{\alpha} d_H(M_{\alpha}, N_{\alpha})\ge d_H(M, N).$

Assume that $\liminf\limits_{\alpha} d_H(M_{\alpha}, N_{\alpha}) < d_H(M, N)$. Since the pseudometric of $X$ is bounded on $B$, then $d_H(M, N) < \infty$. So there exist $\varepsilon > 0$ and a subnet $\bigl\{(M_{\beta}, N_{\beta})\bigr\}$ such that for every $\beta$ we have $d_H(M_{\beta}, N_{\beta}) < d_H(M, N) - 2\varepsilon.$ Hence,
\begin{equation}\label{eq:sup}
\sup_{x\in M_{\beta}} |x\, N_{\beta}| < d_H(M, N) - 2\varepsilon
\end{equation}
for every $\beta$.

Without loss of generality, assume that $d_H(M, N) = \sup_{x\in M} |x\, N|$ (otherwise, interchange $M$ and $N$). By the definition of supremum, for the chosen $\varepsilon$, there exists $m\in M$ such that $d_H(M, N) - |m\, N| < \varepsilon,$ that is,
\begin{equation}\label{eq:dH}
|m\, N| > d_H(M, N) - \varepsilon.
\end{equation}

Since $\tau$ is a compact Hausdorff topology and the net $\{M_{\beta}\}$ converges to $M$ with respect to the topology $V_{\tau}$, it follows from Theorem~\ref{thm:Kuratow} that $\Lim\limits_{\beta} M_{\beta} = M$. So, for the point $m$, there exist points $m_{\beta}\in M_{\beta}$ such that the net $\{m_{\beta}\}$ converges to $m$ with respect to the topology $\tau$. By~(\ref{eq:sup}), we have
$$|m_{\beta}\, N_{\beta}| < d_H(M, N) - 2\varepsilon.$$

Since $(B, \tau)$ is compact and the $N_{\beta}$ are $\tau$-closed, the $N_{\beta}$ are $\tau$-compact. Together with the lower semicontinuity of the pseudometric of $X$ on $B$ in the topology $\tau\times \tau$, this yields the existence of points $n_{\beta}\in N_{\beta}$ such that
\begin{equation}\label{eq:d_H_eps}
|m_{\beta}\, n_{\beta}| = |m_{\beta}\, N_{\beta}| < d_H(M, N) - 2\varepsilon.
\end{equation}

By the compactness of $(B, \tau)$, the net $\{n_{\beta}\}$ has a subnet $\{n_{\gamma}\}$ which converges in $\tau$ to some $n\in B$. The net $\{N_{\gamma}\}$ converges to $N$ in the topology $V_{\tau}$, therefore $\Lim\limits_{\gamma} N_{\gamma} = N$ by Theorem~\ref{thm:Kuratow}. Since $n_{\gamma}\in N_{\gamma}$, by the definition of the Kuratowski limit, we have $n\in N$. Thus, we have $(m_{\gamma}, n_{\gamma})\to (m, n)$. From~\eqref{eq:d_H_eps} and from the lower semicontinuity of the pseudometric of $X$ on $B$ in the topology $\tau\times \tau$, it follows that
$$|m\, n|\le \liminf\limits_{\gamma} |m_{\gamma}\, n_{\gamma}| \le d_H(M, N) - 2\varepsilon.$$
Since $n\in N$, we have $|m\, N|\le |m\, n|\le d_H(M, N) - 2\varepsilon.$
But this contradicts~(\ref{eq:dH}).

Hence, $\liminf\limits_{\alpha} d_H(M_{\alpha}, N_{\alpha})\ge d_H(M, N)$, and therefore $d_H$ is lower semicontinuous on $\mathcal{P}_{\Cl_{\tau}}(B)\times \mathcal{P}_{\Cl_{\tau}}(B)$ in the topology $V_{\tau}\times V_{\tau}$.\qed

\end{proof}

\begin{definition}
Let $X$ be a pseudometric space and $\tau$ be a topology on $X$. Then $X$ is called \textit{compactly equipped with a topology} $\tau$ (or $\tau$-\textit{compactly equipped}) if the following conditions hold$:$
\begin{itemize}
\item each closed ball $B\subset X$ is $\tau$-compact$;$
\item for each closed ball $B\subset X$, the pseudometric of $X$ is lower semicontinuous on $B\times B$ with respect to the product topology $\tau\times \tau.$
\end{itemize}
\end{definition}

\begin{theorem}\label{thm:hyperspace}
Let $X$ be a pseudometric space and let a closed ball $B$ in $X$ be equipped with its own topology $\tau$ such that$:$
\begin{itemize}
\item $(B, \tau)$ is a compact Hausdorff topological space$;$
\item the pseudometric of $X$ is lower semicontinuous on $B\times B$ with respect to the product topology $\tau\times \tau$.
\end{itemize}
Then $\mathcal{P}_{\Cl_{\tau}}(B)$ is a $V_{\tau}$-compactly equipped $($and hence, compactly equipped$)$ space.
\end{theorem}

\begin{proof}

According to Theorem~\ref{thm:dH_semicont}, the Hausdorff distance $d_H$ is lower semicontinuous on $\mathcal{P}_{\Cl_{\tau}}(B)\times \mathcal{P}_{\Cl_{\tau}}(B)$ with respect to $V_{\tau}\times V_{\tau}$.

Let $\mathbb{B}$ be a closed ball in $\mathcal{P}_{\Cl_{\tau}}(B)$. Since $d_H$ is lower semicontinuous on $\bigl(\mathcal{P}_{\Cl_{\tau}}(B), V_{\tau}\bigr)$ and $\mathbb{B}$ is a closed ball of $\mathcal{P}_{\Cl_{\tau}}(B)$, then by Proposition~\ref{prop:cont_level_0} the ball $\mathbb{B}$ is $V_{\tau}$-closed. Theorem~\ref{thm:comp_hyp} implies that $\bigl(\mathcal{P}_{\Cl_{\tau}}(B), V_{\tau}\bigr)$ is a compact space. Hence, $\mathbb{B}$ is $V_{\tau}$-compact in $\mathcal{P}_{\Cl_{\tau}}(B)$.\qed

\end{proof}

Theorem~\ref{thm:hyperspace} and Theorem~\ref{thm:Ambrosio} imply the following corollary.

\begin{corollary}\label{cor:hyperspace_param_net}
Let $X$ be a pseudometric space and let a closed ball $B$ in $X$ be equipped with its own topology $\tau$ such that$:$
\begin{itemize}
\item $(B, \tau)$ is a compact Hausdorff topological space$;$
\item the pseudometric of $X$ is lower semicontinuous on $B\times B$ with respect to the product topology $\tau\times \tau$.
\end{itemize}
Then, for any finite set $\mathcal{A}\subset \mathcal{P}_{\Cl_{\tau}}(B)$ and any connected graph $G = (V, E)$ with boundary $\mathcal{A}$, there exists a minimal parametric network of type $G$ joining $\mathcal{A}$.
\end{corollary}

\begin{corollary}\label{cor:hyperspace_St}
Let $X$ be a pseudometric space and let a closed ball $B$ in $X$ be equipped with its own topology $\tau$ such that$:$
\begin{itemize}
\item $(B, \tau)$ is a compact Hausdorff topological space$;$
\item the pseudometric of $X$ is lower semicontinuous on $B\times B$ with respect to the product topology $\tau\times \tau$.
\end{itemize}
Then, for any finite set $\mathcal{A}\subset \mathcal{P}_{\Cl_{\tau}}(B)$, there exists a minimal Steiner tree joining $\mathcal{A}$.
\end{corollary}

%%%%%%%%%%%%%%%%%%%%%%%%%%%%%%
\subsection{Existence Theorems in Hyperspaces Over Dual Spaces}\label{sec:hyp_dual}
%%%%%%%%%%%%%%%%%%%%%%%%%%%%%%

The following theorem is a direct consequence of Theorem~\ref{thm:hyperspace}.

\begin{theorem}\label{thm:hyperspace_param_net_tau}
Let $X$ be a pseudometric space and $\tau$ be a Hausdorff topology on $X$ such that $X$ is $\tau$-compactly equipped. Then $\mathcal{P}_{\Cl_{\tau}, \B}(X)$ is a $V_{\tau}$-compactly equipped $($and hence, compactly equipped$)$ space.
\end{theorem}

\begin{proof}

Let $\mathbb{B}$ be a closed ball in $\mathcal{P}_{\Cl_{\tau}, \B}(X)$. According to Lemma~\ref{lem:ball}, there is a closed ball $B\subset X$ such that $\mathbb{B}$ is a closed ball in $\mathcal{P}_{\Cl_{\tau}}(B)$. By Theorem~\ref{thm:hyperspace}, the hyperspace $\mathcal{P}_{\Cl_{\tau}}(B)$ is a $V_{\tau}$-compactly equipped space. Hence, $\mathbb{B}$ is $V_{\tau}$-compact and the Hausdorff distance is lower semicontinuous on $\mathbb{B}\times \mathbb{B}$ with respect to the product topology $V_{\tau}\times V_{\tau}$.\qed

\end{proof}

Note that by definition, the weak$^*$ topology is the weakest topology on $X^*$ in which all evaluation functionals from $X$ are continuous. As we have shown in Section~\ref{sec:ref_space}, the evaluation functionals are continuous with respect to the operator norm of the dual space. Consequently, the weak$^*$ topology, as a collection of subsets, is contained in the metric topology of the dual space. Since the Hausdorff distance $d_H$ is a metric on the set of nonempty closed (in a metric sense) and bounded subsets of any metric space (in particular, dual), on the set of nonempty weak$^*$-closed and bounded subsets of the dual space, that distance $d_H$ is also a metric.

By Corollary~\ref{cor:Haus_loc}, the weak$^*$ topology is a Hausdorff topology on $X^*$ and $X^*$ is weak$^*$-compactly equipped according to Theorem~\ref{thm:dual}. Thus, Theorem~\ref{thm:hyperspace_param_net_tau} gives us the following theorem.

\begin{theorem}\label{thm:hyp_tau_comp}
Let $X$ be a dual space and $\tau$ be the weak$^*$ topology on $X$. Then $\mathcal{P}_{\Cl_{\tau}, \B}(X)$ is a $V_{\tau}$-compactly equipped $($and hence, compactly equipped$)$ metric space.
\end{theorem}

Theorem~\ref{thm:hyp_tau_comp} and Theorem~\ref{thm:Ambrosio} give us the following corollary.

\begin{corollary}\label{cor:dual_hyperspace}
Let $X$ be a dual space and $\tau$ be the weak$^*$ topology on $X$. Then for any finite set $\mathcal{A}\subset \mathcal{P}_{\Cl_{\tau}, \B}(X)$ and any connected graph $G = (V, E)$ with boundary $\mathcal{A}$, there exists a minimal parametric network of type $G$ joining $\mathcal{A}$.
\end{corollary}

\begin{corollary}\label{cor:dual_hyperspace_St}
Let $X$ be a dual space and $\tau$ be the weak$^*$ topology on $X$. Then for any finite set $\mathcal{A}\subset \mathcal{P}_{\Cl_{\tau}, \B}(X)$, there exists a minimal Steiner tree joining $\mathcal{A}$.
\end{corollary}

We now give a brief example-remark. Before formulating it, let us say that a topological space is called \textit{locally compact} if every point has a neighbourhood whose closure is compact. In this example-remark we shall also use such notions as \textit{finite signed measure} and \textit{complex measure}, which are defined, for instance, in~\cite{Cohn}, Subsection~4.1. A measure whose domain is the Borel $\sigma$-algebra of some topological space is called a \textit{Borel measure}. The notion of finite signed (or complex) Borel \textit{regular measure}, which will also be needed below, is introduced in~\cite{Cohn}, Subsection~7.3. Also, $C_0(X)$ will denote the Banach space of all continuous functions on $X$ such that for every $f\in C_0(X)$ and every $\varepsilon > 0$ there exists a compact set outside which $|f| < \varepsilon$. Such functions $f$ are said to \textit{vanish at infinity}. Accordingly, by $C_0^{\mathbb{C}}(X)$ we shall mean the space of all continuous complex-valued functions on $X$ that vanish at infinity. Also Cohn in the book~\cite{Cohn} denotes by $M_r(X,\mathbb{R})$ the set of all finite signed regular Borel measures on $X$ and by $M_r(X,\mathbb{C})$ the set of all complex regular Borel measures on $X$. We shall need the following theorem.

\begin{theorem}[{\cite[Theorem~7.3.6]{Cohn}}]\label{thm:cohn}
Let $X$ be a locally compact Hausdorff space. Then the map that takes the finite signed $($or complex$)$ regular Borel measure $\mu$ to the functional 
$$F\colon C_0(X)\rightarrow \mathbb{R}\,\,\, \bigl(F\colon C_0^{\mathbb{C}}(X)\rightarrow \mathbb{C}\bigr),$$
$$F\colon f\mapsto \int_X f \,d\mu$$
is an isometric isomorphism of the Banach space $M_r(X,\mathbb{R})$ $\bigl($or $M_r(X,\mathbb{C})\bigr)$ onto the dual of the Banach space $C_0(X)$ $\bigl($or $C_0^{\mathbb{C}}(X)\bigr)$.
\end{theorem}

\begin{remark}
Thus, by virtue of Theorem~\ref{thm:cohn}, Corollary~\ref{cor:dual_equip}, and Theorem~\ref{thm:Ambrosio}, the Steiner problem of finding minimal (parametric) networks in the space $M_r(X,\mathbb{R})$ (or $M_r(X,\mathbb{C})$) is solvable, where $X$ is a locally compact Hausdorff space. Moreover, by Corollaries~\ref{cor:dual_hyperspace} and~\ref{cor:dual_hyperspace_St}, the Steiner problem is also solvable in the hyperspace of all nonempty weak$^*$-closed and bounded subsets of the Banach space $M_r(X,\mathbb{R})$ (or $M_r(X,\mathbb{C})$).
\end{remark}

%%%%%%%%%%%%%%%%%%%%%%%%%%%%%%
\subsection{Existence Theorems in Hyperspaces Over Reflexive Spaces}\label{sec:hyp_ref}
%%%%%%%%%%%%%%%%%%%%%%%%%%%%%%

\begin{lemma}\label{lem:hyp_closed}
Let $X$ be a dual space, $B$ be a closed ball in $X$ and $\tau$ be the weak$^*$ topology on $X$. Then $\mathcal{P}_{\Cl_{\tau}, \Conv}(B)$ is $V_{\tau}$-closed in $\mathcal{P}_{\Cl_{\tau}}(B)$.
\end{lemma}

\begin{proof}

Consider a net $\{S_{\alpha}\}_{\alpha\in \Lambda}\subset \mathcal{P}_{\Cl_{\tau}, \Conv}(B)$ which is convergent to $S\in \mathcal{P}_{\Cl_{\tau}}(B)$ with respect to $V_{\tau}$. Show that for any $x, y\in S$ and $t\in (0, 1)$, it holds $tx + (1-t)y \in S$. 

By Corollary~\ref{cor:Haus_loc} and Theorem~\ref{thm:Ban_Ala}, the weak$^*$ topology $\tau$ is a compact Hausdorff topology on $B$, therefore $\Lim\limits_{\alpha} S_{\alpha} = S$ by Theorem~\ref{thm:Kuratow}. So, for points $x, y\in S$, there exist points $x_{\alpha}, y_{\alpha}\in S_{\alpha}$ such that $x_{\alpha}\to x$ and $y_{\alpha}\to y$ with respect to $\tau$. The weak$^*$ topology is linear, therefore 
$$z_{\alpha} = tx_{\alpha} + (1-t)y_{\alpha} \to z = tx + (1-t)y.$$
Since each $S_{\alpha}$ is convex, $z_{\alpha}\in S_{\alpha}$. By the definition of the Kuratowski limit, it means that $z\in \Lim\limits_{\alpha} S_{\alpha} = S$. Thus, $S\in \mathcal{P}_{\Cl_{\tau}, \Conv}(B)$ and so $\mathcal{P}_{\Cl_{\tau}, \Conv}(B)$ is $V_{\tau}$-closed in $\mathcal{P}_{\Cl_{\tau}}(B)$.\qed

\end{proof}

\begin{theorem}\label{thm:conv_hyperspace}
Let $X$ be a reflexive space and $\tau$ be the weak$^*$ topology on $X$. Then $\mathcal{P}_{\Cl, \B, \Conv}(X)$ is a $V_{\tau}$-compactly equipped $($and hence, compactly equipped$)$ metric space.
\end{theorem}

\begin{proof}

By Corollary~\ref{cor:equal}, we have $\mathcal{P}_{\Cl, \B, \Conv}(X) = \mathcal{P}_{\Cl_{\tau}, \B, \Conv}(X)$. Let $\mathbb{B}$ be a closed ball in $\mathcal{P}_{\Cl, \B, \Conv}(X) = \mathcal{P}_{\Cl_{\tau}, \B, \Conv}(X)$. According to Lemma~\ref{lem:ball}, there is a closed ball $B\subset X$ such that $\mathbb{B}$ is a closed ball in $\mathcal{P}_{\Cl_{\tau}, \Conv}(B)$. Note that for $\mathbb{B}\subset \mathcal{P}_{\Cl_{\tau}, \Conv}(B)$ there exists a closed ball $\mathbb{B}'\subset \mathcal{P}_{\Cl_{\tau}}(B)$ such that $\mathbb{B} = \mathbb{B}'\cap \mathcal{P}_{\Cl_{\tau}, \Conv}(B)$. By Theorem~\ref{thm:hyperspace}, the hyperspace $\mathcal{P}_{\Cl_{\tau}}(B)$ is $V_{\tau}$-compactly equipped and so $\mathbb{B}'$ is $V_{\tau}$-compact.

From Corollary~\ref{cor:Haus_loc}, Theorem~\ref{thm:Ban_Ala} and Theorem~\ref{thm:comp_hyp}, it follows that $\mathcal{P}_{\Cl_{\tau}}(B)$ is a $V_{\tau}$-compact Hausdorff space. According to Lemma~\ref{lem:hyp_closed}, the space $\mathcal{P}_{\Cl_{\tau}, \Conv}(B)$ is $V_{\tau}$-closed in $\mathcal{P}_{\Cl_{\tau}}(B)$. Hence, $\mathbb{B}$ is $V_{\tau}$-compact in $\mathcal{P}_{\Cl_{\tau}}(B)$ and so in $\mathcal{P}_{\Cl_{\tau},\Conv}(B)$ and so in $\mathcal{P}_{\Cl_{\tau}, \B, \Conv}(X) = \mathcal{P}_{\Cl, \B, \Conv}(X)$.

According to Theorem~\ref{thm:dH_semicont}, the Hausdorff distance $d_H$ is lower semicontinuous on $\mathbb{B}\times \mathbb{B}\subset \mathcal{P}_{\Cl_{\tau}}(B)\times \mathcal{P}_{\Cl_{\tau}}(B)$ with respect to $V_{\tau}\times V_{\tau}$. Thus, $\mathcal{P}_{\Cl, \B, \Conv}(X)$ is a $V_{\tau}$-compactly equipped metric space.\qed

\end{proof}

Theorem~\ref{thm:conv_hyperspace} and Theorem~\ref{thm:Ambrosio} give us the following corollary.

\begin{corollary}\label{cor:Ambrosio_ref}
Let $X$ be a reflexive space. Then, for any finite set $\mathcal{A}\subset \mathcal{P}_{\Cl, \B, \Conv}(X)$ and for any connected graph $G = (V, E)$ with boundary $\mathcal{A}$, there exists a minimal parametric network of type $G$ joining $\mathcal{A}$.
\end{corollary}

\begin{corollary}\label{cor:Ambrosio_St_ref}
Let $X$ be a reflexive space. Then, for any finite set $\mathcal{A}\subset \mathcal{P}_{\Cl, \B, \Conv}(X)$, there exists a minimal Steiner tree joining $\mathcal{A}$.
\end{corollary}

%
% ---- Bibliography ----
%

\end{document}